\documentclass[12pt]{article}

\usepackage[T1]{fontenc}
\usepackage{avant}
\usepackage[cp1250]{inputenc}
\usepackage{amsmath,amssymb,amsthm}
\usepackage{geometry}
\usepackage{caption}
\usepackage{float}
\usepackage{graphicx}
\usepackage{dynkin-diagrams}
\usepackage{makeidx}
\usepackage[matrix,arrow,curve]{xy}
\usepackage{boxedminipage}
\usepackage{epic}
\usepackage{longtable}
\usepackage{ifthen}
\usepackage{tabularx}
\usepackage{parskip}
\usepackage{verbatim}
\usepackage{multirow}

\usepackage[ruled,linesnumbered,norelsize]{algorithm2e}
\usepackage{authblk}

\newtheorem{theorem}{Theorem}
\newtheorem{proposition}{Proposition}

\newtheorem{definition}{Definition}

\newtheorem{corollary}{Corollary}
\newtheorem{remark}{Remark}

\newtheorem{question}{Question}
\newtheorem{problem}{Problem}

\newcommand{\g}{\mathfrak{g}}
\newcommand{\ssl}{\mathfrak{sl}}
\newcommand{\h}{\mathfrak{h}}
\newcommand{\myl}{\mathfrak{l}}
\newcommand{\z}{\mathfrak{z}}
\newcommand{\myc}{\mathfrak{c}}
\renewcommand{\a}{\mathfrak{a}}
\newcommand{\p}{\mathfrak{p}}
\newcommand{\myk}{\mathfrak{k}}
\newcommand{\SL}{\mathrm{\mathop{SL}}}
\newcommand{\R}{\mathbb{R}}
\newcommand{\HH}{\mathcal{H}}

\usepackage{color}
\newcommand{\comm}{\textcolor{blue}}

\title{Non-virtually abelian discontinuous group actions vs.  proper $SL(2,\mathbb{R})$-actions on  homogeneous spaces}
\author{Maciej Boche\'nski,  Willem A. de Graaf, Piotr Jastrz\c ebski,\\
  Aleksy Tralle}

\begin{document}

\maketitle
\abstract{We develop algorithms and computer programs which verify criteria of properness of discrete group actions on semisimple homogeneous spaces. We apply these algorithms to find new examples of non-virtually abelian discontinuous group actions on homogeneous spaces which do not admit proper $SL(2,\mathbb{R})$-actions.}
\vskip10pt

\section{Proper actions of non-virtually abelian discrete subgroups}

Assume that $G$ is a linear connected semisimple real Lie group and $H\subset G$ is a closed
reductive subgroup of $G$ with finitely many connected components such that $G/H$
is non-compact. Denote by $\mathfrak{g},\mathfrak{h}$ the Lie algebras of $G, H$ respectively. Recall that a
discrete group is non-virtually abelian if it does not contain an abelian subgroup of
finite index. Consider the following three conditions:

\begin{itemize}
\item C1 := the space $G/H$ admits a properly discontinuous action of an infinite discrete subgroup of $G$,
\item C2 := the space $G/H$ admits a properly discontinuous action of a non-virtually
abelian infinite discrete subgroup of $G$,
\item C3 := the space $G/H$ admits a proper action of a subgroup $L\subset G$ locally
isomorphic to $SL(2,\mathbb{R})$.
\end{itemize}

We will say that $G/H$ is a C$i$ space, $i=1,2,3$ if $G/H$ fulfills the condition C$i$. If
$H$ is compact then every closed subgroup of $G$ acts \comm{properly} on $G/H$ and all  the three
 conditions above hold. However this is not the case for non-compact $H$. For instance we have the Calabi-Markus phenomenon  (see, for example \cite[Corollary 4.4]{K89}):

$$G/H \textrm{ is C1 if and only if } \operatorname{rank}_{\mathbb{R}}\mathfrak{g} > \operatorname{rank} _{\mathbb{R}}\mathfrak{h}.$$

\noindent One of the reasons for considering these properties comes from the idea of the "continuous analogue" due to T. Kobayashi: under the above mentioned assumptions a homogeneous space $G/H$ admits a proper action of an infinite discontinuous group if it admits a proper action of a one-dimensional non-compact closed subgroup in  $G$. Going along this line of thinking, it was shown in \cite{ok} that a semisimple symmetric space admits a non-virtually abelian properly discontinuous action of a subgroup $\Gamma\subset G$ if and only if it admits a proper action of a three-dimensional non-compact simple Lie subgroup of $G$. Thus, in our terminology, for the class of symmetric spaces, C2 and C3 are equivalent (\cite[Theorem 2.2]{ok}).  Some other classes of homogeneous spaces of class C3 were described in \cite{btjo} and of class C2 in \cite{bjt-8}.
In \cite{ben}  a necessary and sufficient condition for $G/H$ to satisfy C2 was found. Let us explain it in some detail. 
Let $K\subset G$ be  a maximal compact subgroup of $G$ and denote by $\mathfrak{k}$ the Lie algebra
of $K$. Let
$\mathfrak{g}=\mathfrak{k}+\mathfrak{p}$
be a Cartan decomposition of $\mathfrak{g}$ induced by a Cartan involution fixing $\mathfrak{k}$ and
choose a maximal abelian subspace $\mathfrak{a}$ of $\mathfrak{p}$. The little Weyl group is defined as
$W:= N_K (\mathfrak{a}) /Z_K (\mathfrak{a})$ where $N_K (\mathfrak{a})$ and  $Z_K (\mathfrak{a})$ denote respectively the normalizer and the centralizer
of $\mathfrak{a}$ in $K$. We fix  a positive system of the restricted root system of $\mathfrak{g}$ determined by $\mathfrak{a}$ and denote by $\mathfrak{a}^{+}$ the (closed) positive Weyl chamber. The finite group $W$ acts on $\mathfrak{a}$  by orthogonal transformations as a finite  group generated by reflections in the hyperplanes determined by simple roots of the  restricted root
system of $\mathfrak{g}$. Let $w_0\in W$ be the longest element and put

$$\mathfrak{b}^{+} := \{ X\in\mathfrak{a}^{+} |  -w_0 (X) = X\},$$
$$\mathfrak{b} := \textrm{Span}_{\mathbb{R}}(\mathfrak{b}^{+}).  \label{eq:defb}$$

We may assume that $\theta|_{\mathfrak{h}}$ is a Cartan involution of $\mathfrak{h}$ so that $\mathfrak{h} = \mathfrak{k}_{\mathfrak{h}}+\mathfrak{p}_{\h}$ is a Cartan
decomposition of $\mathfrak{h}$ such that $\mathfrak{k}_{\mathfrak{h}} \subset \mathfrak{k}, \mathfrak{p}_{\h}\subset \mathfrak{p}$. We also may assume that a maximal
abelian subspace $\mathfrak{a}_{\h}$ of $\mathfrak{p}_{\h}$ belongs to $\mathfrak{a}$ (see \cite{K89}). The  following theorem yields the required criterion:

\begin{theorem}[\cite{ben}, Theorem 1]\label{thm:benois} $G/H$ is C2 if and only if for every $w$ in $W$, $w\cdot \mathfrak{a}_{\h}$
does not contain $\mathfrak{b}$. In this case, one can choose a discrete subgroup of $G$ acting properly discontinuously on $G/H$ to be free and Zariski dense in $G$.
\end{theorem}

For example the space $SL(3,\mathbb{R})/SL(2,\mathbb{R})$ fulfills the condition C1 but does not fulfill
the condition C2 and therefore the condition C3 (see \cite[Example 1]{ben}).   
For many important classes of homogeneous spaces C2 and C3 are equivalent. For example, apart from irreducible symmetric spaces, this holds  for some  strongly regular homogeneous spaces (\cite[Theorem 2 and Corollary 1]{b}).
On the other hand,
there are known exactly two examples of spaces which fulfill the condition C2 but
not C3 \cite{b,ok1}. Thus for a space of reductive type we only have the following:
$$C1 \Leftarrow C2 \Leftarrow C3.$$
\noindent The criterion for $G/H$ to be C3 is given in \cite{K89} in a general setting of the properness of the action of a Lie subgroup $L\subset G$ on $G/H$. We describe it in Subsection \ref{subsec:kob}. Studying conditions C2, C3 and relations between them is of great importance. We refer to a recent paper \cite{K022} for a more detailed account. However, let us mention two challenging problems from \cite{K022}.
\begin{problem}\label{prob:c-k} Determine all pairs $(G,H)$ such that $G/H$  admits co-compact discontinuous actions of discrete subgroups $\Gamma\subset G$.
\end{problem}
 \noindent This is a long standing open problem \cite{ko, kod, K89}. There are many partial results, for example, \cite{bt1, bt2, bjt, bjstw, mor}, however, the general case is not settled. 
\begin{problem}\label{prob:c3} Find a necessary and sufficient condition for $G/H$ to admit a discontinuous action of a group $\Gamma\subset G$ isomorphic to a surface group $\pi_1(\Sigma_g)$ with $g\geq 2$.
\end{problem}
\noindent Note that this problem is the same as C3 for the class of symmetric spaces \cite{ok}. The general case is still open. Some partial results on homogeneous spaces which are C3 are obtained in \cite{btjo}.
Also, in this context an important question arises.
\begin{question}\label{quest:c2relc3} What is the relation between classes C2 and C3? 
\end{question}
\noindent The known two examples of homogeneous spaces $G/H$ which are C2 but not C3 are obtained in \cite{b} and \cite{ok1} by applying the criterion of the properness of the Lie subgroup action on $G/H$ (see Subsection \ref{subsec:kob}) and Theorem \ref{thm:benois}.
 One shows that there is
an appropriate subspace $\mathfrak{a}_{m} \subset \mathfrak{a}$ such that $W\mathfrak{a}_m$ does not contain $\mathfrak{b}$ and that (by
Kobayashi's criterion of proper actions and the classification of nilpotent orbits) for
any $\mathfrak{sl}(2,\mathbb{R}) \hookrightarrow \mathfrak{g}$ there exists $g\in G$ such that

$$\operatorname{Ad} _g (\mathfrak{sl}(2,\mathbb{R}) ) \cap \mathfrak{a}_m \neq \{0\}.$$

\noindent In this case $G/A_{m},$ where the subgroup $A_{m}\subset G$ corresponds to $\mathfrak{a}_{m} ,$  is a C2 space but not a C3 space. Notice that these examples are quotients  obtained by explicit calculations. The calculations are fairly easy because in these cases the description of the embedding $\mathfrak{h}\hookrightarrow \mathfrak{g}$ in terms of the root systems is simple. This description enables one to write down the action of the little Weyl group on $\mathfrak{a}_{m}$ and use the known theoretical results.  Thus, these examples are not sufficient for a qualitative understanding of the relation between C2 and C3. On the other hand, one can see  that important Problems \ref{prob:c-k} and \ref{prob:c3} theoretically are settled in terms of root systems and the action of the  little Weyl group, by Theorems \ref{thm:benois} and \ref{thm:kobcrit}. The known examples of $G/H$ which are C2 but not C3 are obtained in the same way. Therefore, from the computational point of view these become problems which could and should be attacked in an algorithmic fashion. We pose the following.

\begin{problem}\label{quest:c2vsc3}  Create  computer algorithms and programs which verify C2 and C3 for general semisimple homogeneous spaces.
\end{problem}
\noindent Our  aim is to solve this problem. We create algorithms verifying C2 versus C3 and implement them in the computational algebra system {\sf GAP4} \cite{gap}.
One can see that Theorem \ref{thm:benois} yields a sufficient condition for the {\it non-existence} of compact quotients $\Gamma\backslash G/H$. Our approach yields methods of computer checking this as well. Note that a substantial amount of work related to the algorithms and computer programs for calculations in Lie algebras and a classification of semisimple  subalgebras in semisimple Lie algebras relevant to this article was done in \cite{ckforms}, \cite{fd}, \cite{dg1}, \cite{dg}. 

The main application of our solution to Problem \ref{quest:c2vsc3} is finding new examples of semisimple homogeneous spaces which are C2 but not C3. Consider the class of homogeneous spaces $G/H$, where $G$ is connected,  linear and \comm{absolutely} simple, and $H$ is a maximal proper semisimple subgroup of $G$ and assume that $\mathfrak{h}^c$ is maximal in $\mathfrak{g}^c$ and $\operatorname{rank} G\leq 8$ (here $\mathfrak{h}^c$,
$\mathfrak{g}^c$ denote the complexifications of $\mathfrak{h}$,
$\mathfrak{g}$ respectively). 
The corresponding pairs $(\mathfrak{g},\mathfrak{h})$ are classified in \cite{dg}. \comm{More precisely, for each complex embedding of a maximal semisimple subalgebra $\mathfrak{h}^c\subset \mathfrak{g}^c$ and each real form $\mathfrak{g}$ of $\mathfrak{g}^c$ in \cite{dg}
  the real forms $\mathfrak{h}$ of $\mathfrak{h}^c$ are determined such that
  there exists an embedding $\mathfrak{h}\subset \mathfrak{g}$. If $\mathfrak{h}$ has a nontrivial centralizer in $\mathfrak{g}$ then this centralizer is also
  given, so that the resulting list consists of reductive subalgebras. The
  resulting database of embeddings of real forms of maximal semisimple subalgebras determined in \cite{dg} will be denoted by
  $\mathcal{G}\mathcal{M}$ in this paper.}
\begin{theorem}\label{thm:main} For any $(\mathfrak{g},\mathfrak{h})\in\mathcal{G}\mathcal{M}$ such that $\operatorname{rank}\mathfrak{g}\leq 6$ conditions C2 and C3 are equivalent. There are two new  examples of $(\mathfrak{g},\mathfrak{h})\in\mathcal{G}\mathcal{M}$ with split $\mathfrak{g}$ and $\mathfrak{h}$ such that $\operatorname{rank}\mathfrak{g}=7,8$ which belong to C2 but not C3:
$$(\mathfrak{e}_{7(7)},\mathfrak{sl}(2,\mathbb{R})\oplus\mathfrak{f}_{4(4)}), (\mathfrak{e}_{8(8)},\mathfrak{f}_{4(4)}\oplus\mathfrak{g}_{2(2)}).$$  
\end{theorem}

This theorem is proved by computation, using the implementation of the
algorithms given in this paper. The arxiv version of this paper (available
on \verb+https://arxiv.org/abs/2206.01069+) has an
ancillary file which is a short {\sf GAP} program that checks that the
two homogeneous spaces listed in the theorem indeed do not satisfy C3.

Notice that the assumption of $H$ being maximal in $G$ is very interesting in the context of searching for homogeneous spaces which are $C2$ but not $C3.$ On  one hand, one should expect that homogeneous spaces which are not C3 are given by ``large'' semisimple subgroups of $G$ (to be more specific, by subgroups which have as large real rank as possible). On the other hand we are restricted by the condition C2 (so we cannot take $\operatorname{rank}_{\mathbb{R}}\mathfrak{g} = \operatorname{rank}_{\mathbb{R}}\mathfrak{h}$).

Finally, let us stress that the main results of this article are algorithms and their implementation which check C2 and  C3, as well as Theorem \ref{thm:main}. These use theoretical results contained in Theorem \ref{thm:benois} (Benoist), Theorem \ref{thm:kobcrit} (Kobayashi), Theorem \ref{the5}, Proposition \ref{thm:okuda} (Okuda) and Theorem \ref{thm:c3} which follows from Theorem \ref{the5} and which yields a method of computer checking whether $G/H$ is C3.

\section{Lie theory approach}

Let us note that our notation and terminology is close to the sources \cite{ov2, OV}. Throughout this section $G$ is a real reductive linear connected Lie group with the Lie algebra $\mathfrak{g}.$

\subsection{$\mathfrak{sl}(2,\mathbb{R})$-triples and antipodal hyperbolic orbits}
  We say that an element $X \in \mathfrak{g}$ is {\it hyperbolic}, if $X$ is semisimple (that is, $ad_{X}$ is diagonalizable) and all eigenvalues of $ad_{X}$ are real.
\begin{definition}
{\rm An adjoint orbit $O_{X}:=Ad_{G}X$ is said to be hyperbolic if $X$ (and therefore every element of $O_{X}$) is hyperbolic. An orbit $O_{Y}$ is antipodal if $-Y\in O_{Y}$ (and therefore for every $Z\in O_{Y},$ $-Z\in O_{Y}$).} 
\end{definition} 
\begin{proposition}[\cite{ok}, Proposition 4.5]
Let $X\in \mathfrak{g}$ be hyperbolic. The intersection of the complex hyperbolic orbit of $X$ in $\mathfrak{g}^{c}$ with $\mathfrak{g}$ is a single adjoint orbit $O_{X}$ in $\mathfrak{g}.$ 
\label{thm:okuda}
\end{proposition}

A triple $(h,e,f)$ of vectors in $\mathfrak{g}$ is called an \textbf{{\it $\mathfrak{sl}(2,\mathbb{R})$-triple}} if
$$[h,e]=2e, \ \ [h,f]=-2f, \ \mathrm{and} \ [e,f]=h.$$
In what follows we will write $\mathfrak{sl}_2$-triple instead of $\mathfrak{sl}(2,\mathbb{R})$-triple, and call $h$ a neutral element.
One can show that $e$ is nilpotent and the adjoint orbit $Ad_{G}h$ is antipodal and hyperbolic (see \cite{cmg}). Furthermore there is the homomorphism of Lie groups
$$\Phi : SL(2,\mathbb{R}) \rightarrow G$$
defined by $d\Phi  \left( \begin{array}{cc} 1 & 0  \\ 0 & -1  \end{array} \right) = h, \ \ d\Phi  
\left( \begin{array}{cc} 0 & 1  \\ 0 & 0  \end{array} \right) = e, \ \ d\Phi  \left( 
\begin{array}{cc} 0 & 0  \\ 1 & 0  \end{array} \right) = f.$

\subsection{Proper $SL(2,\mathbb{R})$-actions and a properness criterion }\label{subsec:kob}

Let $H$ and $L$ be \comm{both} reductive subgroups in $G$. \comm{Let $\theta$ be a Cartan involution of $G$. We denote the differential of $\theta$ by the same letter. We may assume that  $\theta$, as a Cartan involution of $\mathfrak{g}$, has the property $\theta |_{\mathfrak{h}},$ $\theta |_{\mathfrak{l}}$ are Cartan involutions of $\mathfrak{h},$ $\mathfrak{l},$ respectively.} We may also assume that there are maximal abelian subalgebras $\mathfrak{a}$, $\mathfrak{a}_{\h},\mathfrak{a}_{\myl}$ in $\mathfrak{p}$, $\mathfrak{p}_{\h}$ and $\mathfrak{p}_{\myl}$ satisfying the inclusions
$$\mathfrak{a}_{\h}\subset\mathfrak{a},\mathfrak{a}_{\myl}\subset\mathfrak{a}.$$
\begin{theorem}[\cite{K89}, Theorem 4.1] In the above settings, the following conditions are equivalent:
\\ (i) $H$ acts on $G/L$ properly.
\\ (ii) $L$ acts on $G/H$ properly.
\\ (iii) $\mathfrak{a}_{\h} \cap W \mathfrak{a}_{\myl} = \{ 0 \}.$
\label{thm:kobcrit}
\end{theorem}
We can restate these conditions in the language of hyperbolic orbits for the action of $L=SL(2,\mathbb{R}).$
\begin{theorem}[\cite{ok}, Corollary 5.5]\label{the5}
Let $H$ be a reductive subgroup of $G$ and denote by $\mathfrak{h}$ the Lie algebra of $H$. Let $\Phi: SL(2,\mathbb{R}) \rightarrow G$ be a Lie group homomorphism, and denote its differential by $\phi= d \Phi: \mathfrak{sl}(2,\mathbb{R}) \rightarrow \mathfrak{g}$. We put
$$h_{\phi} : = \phi \left( \begin{array}{cc} 1 & 0 \\ 0 & -1
\end{array} \right) \in \mathfrak{g}.$$
Then $SL(2,\mathbb{R})$ acts on $G/H$ properly via $\Phi$ if and only if the real antipodal adjoint orbit through $h_{\phi}$ in $\mathfrak{g}$ does not meet $\mathfrak{a}_{\h}.$
\label{tthh}
\end{theorem}
Notice that, by the classification of nilpotent orbits, the number of subspaces $\mathbb{R}h_{\phi}\subset \mathfrak{a}$ is finite. Also, all $h_{\phi}\in \mathfrak{a}$ are known (see \cite{cmg} for details).

\subsection{a-hyperbolic rank}
Define the {\it a-hyperbolic rank} of a reductive Lie algebra $\mathfrak{g},$ denoted by $\textrm{rank}_{\textrm{a-hyp}}\mathfrak{g},$ as the dimension of $\mathfrak{b}.$ The a-hyperbolic rank of a reductive Lie algebra equals the sum of a-hyperbolic ranks of all its simple ideals. The a-hyperbolic rank of an absolutely simple Lie algebras can be calculated using Table \ref{tab1} (see \cite{bt} for more details about the a-hyperbolic rank).

This invariant will be used in the proof of Theorem \ref{thm:main}.

\begin{center}
 \begin{table}[h]
 \centering
 {\footnotesize
 \begin{tabular}{| c | c |}
   \hline                        
   $\mathfrak{g}$ & $\text{rank}_{\textrm{a-hyp}}$ \\
   \hline
   $\mathfrak{sl}(2k,\mathbb{R})$  & $k$  \\
   \footnotesize $k\geq 1$ & \\
   \hline
   $\mathfrak{sl}(2k+1,\mathbb{R})$  & $k$  \\
   \footnotesize $k\geq 1$ & \\
   \hline
   $\mathfrak{su}^{\ast}(4k)$  & $k$  \\
   \footnotesize $k\geq 1$ & \\
   \hline
   $\mathfrak{su}^{\ast}(4k+2)$  & $k$  \\
   \footnotesize $k\geq 1$ & \\
   \hline
   $\mathfrak{so}(2k+1,2k+1)$  & $2k$  \\
   \footnotesize $k\geq 2$ & \\
   \hline
	 $\mathfrak{e}_{6(6)}$ & $4$  \\
	 \hline
   $\mathfrak{e}_{6(-26)}$  & $1$  \\
   \hline  
 \end{tabular}
 }
\captionsetup{justification=centering}
 \caption{
Real forms of simple Lie algebras $\mathfrak{g}^{c},$ with $\text{rank}_{\mathbb{R}}(\mathfrak{g}) \neq \text{rank}_{\textrm{a-hyp}}(\mathfrak{g}).$
 }
 \label{tab1}
 \end{table}
\end{center}

\section{Checking C3 and C2}\label{sec:c2+c3}

In this section we describe computational methods for checking whether
a given homogeneous space $G/H$ is C3 and C2. \comm{The general theory allows us to work with  their Lie algebras $\g$ and $\h$.} For the use of  weighted Dynkin diagrams one can consult \cite{bt}, \cite{bjt}.

\subsection{Checking C3}

The following theorem follows directly from Theorem \ref{the5}. It underpins
our method for checking whether $G/H$ is C3.

\begin{theorem}\label{thm:c3}
Let $G\supset H$ be semisimple Lie groups with Lie algebras $\g$, $\h$.
Let $\g = \myk\oplus \p$ be a Cartan decomposition of $\g$ and assume (as we may)
that $\h = (\h\cap \myk)\oplus (\h\cap \p)$ is a Cartan decomposition of
$\h$.\textcolor[rgb]{0,0,1}{ Let $\a_{\h}\subset \h\cap\p$ be a maximal abelian subspace of $\h\cap \p$. Let $\a\subset
\p$ be a maximal abelian subspace of $\p $ with $\a_{\h}\subset \a$.} Let
$L\subset G$ be a subgroup isomorphic to $\SL(2,\R)$, with
Lie algebra $\myl$ spanned by $h,e,f$, $h\in\mathfrak{a}$, with the usual commutation relations.
Let $W$ be the little Weyl group of $\g$ with \comm{respect} to $\a$.
Then $G/H$ admits a proper action of $L$ if and only if the orbit $W\cdot h$
has no point in $\a_{\h}$. 
\end{theorem}

We want to check whether there is a subgroup $L$ as in the theorem acting
properly on $G/H$. We assume that we are given the Lie algebras $\g$,
$\h$, along with a Cartan decomposition of $\g$ as in the theorem such that
$\h = \mathfrak{k}_{\mathfrak{h}} \oplus \p_{\h}$, where $\mathfrak{k}_{\mathfrak{h}} = \myk\cap \h$, $\p_{\h} = \p\cap\h$.
We also fix a maximal abelian subspace $\a\subset \p .$

The first problem is to find a conjugate $\a'$ of $\a$ such that
$\a' \cap \h$ is a maximal abelian subspace of $\p_{\h}$.
We do this in the following way.

\begin{enumerate}
\item Let $\a_{\h}\subset \p_{\h}$ be a maximal abelian subspace. Set
  $\hat{\a} = \a_{\h}$ and while $\dim \hat{\a} < \dim \a$ do the following
  \begin{itemize}
  \item Compute the centralizer $\z = \myc_\g( \hat{\a} )$.
  \item Find $x\in \z\cap\p$ with $x\not\in \hat{\a}$ and replace $\hat{\a}$
    by the subalgebra spanned by $\hat{\a}$ and $x$.  
  \end{itemize}
\item Set $\a' = \hat{\a}$.  
\end{enumerate}

We note that this works because any abelian subspace of $\p$ lies in a maximally
abelian subspace of $\p$. So the spaces $\z$ always contain a maximally
abelian subspace.

Now we write $\a$ instead of $\a'$, and assume that $\a_{\h}= \a\cap\h$ is a
maximally abelian subspace of $\p_{\h}$.

We note that the subalgebras of
$\g$ that are isomorphic to $\ssl(2,\R)$ are given by the nilpotent orbits
of $\g$. More precisely, we have the following. The nilpotent orbits in $\g^c$
are parametrized by weighted Dynkin diagrams. Each weighted Dynkin diagram
uniquely determines an element $h$ in a given Cartan subalgebra of $\g^c$.
This $h$ lies in a non-uniquely determined $\mathfrak{sl}_2$-triple $(h,e,f)$.
Let $G^c_{\mathrm{ad}}$ denote the adjoint group of $\g^c$. Then the above
construction gives a bijective correspondence between the list of weighted
Dynkin diagrams and the $G^c_{\mathrm{ad}}$-conjugacy classes of subalgebras
of $\g^c$ that are isomorphic to $\ssl_2$.

By \cite[Proposition 7.8]{ok} we have that a nilpotent orbit in $\g^c$
has points in $\g$ if and only if its weighted Dynkin diagram matches the
Satake diagram of $\g$. Moreover, these are precisely the nilpotent
orbits having $\mathfrak{sl}_2$-triples lying in $\g$. So in the next part of our
procedure we perform the following steps:

\begin{enumerate}
\item Obtain the list of weighted Dynkin diagrams of $\g^c$.
\item For the weighted Dynkin diagrams that match the Satake diagram of
  $\g$ compute an $\mathfrak{sl}_2$-triple $(h,e,f)$ with $h\in \a$. Let $\HH$ denote
  the set of all elements $h$ that are obtained in this way. 
\end{enumerate}

By \cite[Proposition 4.5]{ok} we have that $\HH$ is a set of representatives of
the $G$-orbits of elements $h'\in \g$ that lie in an $\mathfrak{sl}_2$-triple
$(h',e',f')$.

In order to apply Theorem \ref{thm:c3} we need to construct the action of
the little Weyl group $W$ on $\a$. Let $\Sigma\subset \a^*$ denote the set of
roots of $\g$ with respect to $\a$ (see \cite[Chapter 4, \S 4]{OV}).
For each root $\alpha\in \Sigma$ let $P_\alpha = \{ a\in \a\mid \alpha(a)=0\}$.
We define $s_\alpha : \a\to \a$ to be the reflection in the hyperplane
$P_\alpha$, where the inner product on $\a$ is given by the Killing form.
Then the group generated by the $s_\alpha$ is isomorphic to $W$ (\cite[Chapter 4,
  Proposition 4.2]{OV}). 

Now for each $h\in \HH$ we check whether $W\cdot h \cap \a_{\h} =\emptyset$.
If this holds then the subgroup of $G$ corresponding to the subalgebra
spanned by an $\mathfrak{sl}_2$-triple $(h,e,f)$ acts properly on $G/H$.
If the intersection is not empty for all $h\in \HH$ then $G/H$ does not admit a
proper action of a subgroup locally  isomorphic to $\SL(2,\R)$.

\subsection{Checking C2}

Our procedure for checking whether $G/H$ is a C2 space follows directly
from Theorem \ref{thm:benois}. The main ingredients have already been described
above. Indeed, we can compute the little Weyl group $W$ and hence its longest
element. It is then a straightforward task of linear algebra to compute
a basis of the space $\mathfrak{b}$. By running over $W$ we can check
whether there is a $w\in W$ such that $w\cdot \a_{\h}$ contains $\mathfrak{b}$.
The space $G/H$ is C2 if and only if there is no such $w$.

\subsection{Implementation}\label{sec:imp}

We have implemented the two procedures of this section in the computational
algebra system {\sf GAP}4 \cite{gap}, using the package {\sf CoReLG}
\cite{corelg}. This package has a database $\mathcal{G}\mathcal{M}$ which was computed in \cite{dg}. Also we
use the {\sf SLA} package which has lists of weighted Dynkin diagrams for
the nilpotent orbits of the semisimple complex Lie algebras.

Here we comment on two main computational problems that occur. Firstly,
in order to compute the elements of $W$ we need to compute the root system
of $\g$ with respect to the space $\a$. However, in some (but not many) cases
this required eigenvalues that did not lie in the ground field that
{\sf CoReLG} uses. (This field is denoted \verb+SqrtField+; here we do
not go into the details.) So for a few cases our procedures could not be
used.

Secondly, when $\g$ is a split form of type $E_7$ or $E_8$, and the little Weyl
group is equal to the usual Weyl group, the orbits $W\cdot h$ of $h\in\mathcal{H}$ can be very
large. In this case we can handle large orbits by using an algorithm due to Snow
  (\cite{snow}, see also \cite[\S 8.6]{grlie})
  for running over an orbit of the Weyl group without computing all of its
  elements. This algorithm has been implemented in the core system of
  {\sf GAP}4. 
  Also we note two things:
  \begin{itemize}
  \item If we find $h$ such that $W\cdot h \cap \a_{\h}=\emptyset$ then
    we can stop.
  \item If $W\cdot h \cap \a_{\h}\neq \emptyset$ then when enumerating the
    orbit $W\cdot h$ we can reasonably hope to find an element in the
    intersection rather quickly. 
  \end{itemize}
So in practice it usually suffices to enumerate only part of the orbits.

\subsection{Computational results}

The package {\sf CoReLG} contains the database $\mathcal{G}\mathcal{M}$.   
We have used our implementation to check whether $G/H$ is C2 and C3 where
the Lie algebras $\g$, $\h$ are taken from this database (if the subalgebra
$\h$ is reductive rather than semisimple we have taken its derived algebra
instead). We have done this for non-compact $\g$ of rank up to 6, and for
$\g$ equal to the split forms of type $E_7$ and $E_8$. We have obtained the
following results
\begin{itemize}
\item For the simple $\g$ up to rank $6$ there are 842 maximal semisimple subalgebras in
  total. We found no examples of $G/H$ that were C2 but not C3. 
  Of these subalgebras 13 could not be dealt with by our programs for the
  reasons in Section \ref{sec:imp}. We list them in Table \ref{tab}. Analyzing this table, we see that 9 of the spaces from this list do not yield $G/H$ which is C2 but not C3 by theoretical considerations. This is because for such spaces either $\operatorname{rank}_{\mathbb{R}}\mathfrak{g}=\operatorname{rank}_{\mathbb{R}}\mathfrak{h}$, or $\operatorname{rank}_{\mathbb{R}}\mathfrak{h} = 1<\operatorname{rank}_{\textrm{a-hyp}}\mathfrak{g}$ and the result follows from the 
   classification of nilpotent orbits, according to   Section \ref{subsec:filter}. The following homogeneous spaces from Table \ref{tab} are not covered by these arguments:
$$(\mathfrak{sp}(3,3),\mathfrak{sl}(2,\mathbb{R})\oplus \mathfrak{su}(1,3)),(\mathfrak{so}(4,6),\mathfrak{so}(2,3)),(\mathfrak{e}_{6(2)},\mathfrak{sl}(3,\mathbb{R})),(\mathfrak{e}_{6(2)},\mathfrak{g}_{2(2)}).$$
Here we apply one more procedure to verify that these spaces are C3. We will show the following.
\begin{proposition}\label{prop:justif}\comm{ Let $G/H$ be a homogeneous space of absolutely simple non-compact Lie group $G$ and a reductive subgroup $H$. Assume that
  $\operatorname{rank}_{a-hyp}\mathfrak{g}>1$ and  that $\operatorname{rank}_{\mathbb{R}}\mathfrak{h}=2$.} Let $\mathcal{P}$ denote the set of 2-planes in $\mathfrak{a}$ with the following property: $P\in\mathcal{P}$ if and only if it is generated by a pair of vectors $w_1h_1,w_2h_2, w_i\in W$ and $h_i$ are neutral elements of some $\mathfrak{sl}_2$-triple, $i=1,2$. Then $G/H$ is C3 or $\mathfrak{a}_{\h}\in\mathcal{P}$.

\end{proposition} 
\begin{proof}
 \comm{Assume that $G/H$ is not C3 and  $\mathfrak{a}_{\h}\not\in\mathcal{P}$}. 
 \textcolor[rgb]{0,0,1}{If an orbit of a  neutral element of some $\mathfrak{sl}_{2}$-triple does not meet $\mathfrak{a}_{\h}$ then $G/H$ is C3 (by Theorem \ref{tthh}).} Therefore, we may assume that there is one $w_1h_1\in\mathfrak{a}_{\h}.$  Each $W$-orbit of any neutral element must intersect $\mathfrak{a}_{\h}$. It follows that this orbit must intersect the line $\langle w_1h_1\rangle$, because otherwise one would get $\mathfrak{a}_{\h}=\langle w_1h_1,w_2h_2\rangle$ which is not the case. Hence, each $W$-orbit belongs to $\langle w_1h_1\rangle$. Since we can assume that $\mathbb{R}^+(w_1h_1)\subset \mathfrak{a}^+$, where $\mathfrak{a}^+$ is the positive Weyl chamber, we get $\operatorname{rank}_{\operatorname{a-hyp}}\mathfrak{g}=1$ by \cite[Theorem 1.1]{ok2}. Thus,  $\mathfrak{a}_{\h}\in\mathcal{P}$ or $G/H$ is C3. A contradiction.
\end{proof} 
This result implies the following.
\begin{proposition}\label{prop:4-c3}\comm{ Let $G/H$ be a homogeneous space such that 
 $\mathfrak{g}=\mathfrak{sp}(3,3),$ $\mathfrak{g}=\mathfrak{so}(4,6)$ or $\mathfrak{g}=\mathfrak{e}_{6(2)}$ and  $\operatorname{rank}_{\mathbb{R}}\mathfrak{h}=2$. Then  $G/H$  is C3.}
\end{proposition}
\begin{proof}
Note that for the Lie algebras  $\mathfrak{g}=\mathfrak{sp}(3,3),$ $\mathfrak{g}=\mathfrak{so}(4,6)$ and  $\mathfrak{g}=\mathfrak{e}_{6(2)}$ the a-hyperbolic rank is greater than 1.  By Proposition \ref{prop:justif}, it is sufficient to show that for any reductive subgroup $H$ of $G$ with $\mathfrak{a}_{\h}=P\in\mathcal{P}$, the homogeneous space $G/H$ is C3. This is verified by the  procedure described below. It is implemented and executed with the required result. 
\end{proof}
\begin{remark} {\rm Here we do not assume that $(\mathfrak{g},\mathfrak{h})\in\mathcal{GM}$, since the procedure which checks C3 is general and can be applied to all $P\in\mathcal{P}$.}
\end{remark}
Before writing down the procedure for verifying C3 in Proposition 3, let us mention the following. Condition C3 is expressed in terms of the action of the Weyl group acting on vectors which can be read off the weights of the weighted Dynkin diagram. However, the straightforward procedure does not work, because the number of required operations is too big. It is worth to consider the two procedures below.
\vskip6pt
\centerline{\bf Straightforward procedure}
\vskip6pt
1. Write down all weighted Dynkin diagrams for $\mathfrak{g}^c$ which match the Satake diagram for $\mathfrak{g}$.
\vskip6pt
2. Write down vectors $\mathcal{H} = \{ A_1,...,A_t \} \subset\mathfrak{a}$ whose coordinates are read off the weights of the weighted Dynkin diagram  (here $t$ is the number of the non-trivial weighted Dynkin diagrams of $\mathfrak{g}^c$ which match the Satake diagram).
\vskip6pt
3. If some of $A_1,...,A_t$ are on the same line, then choose only one of these, getting $A_1,...,A_k.$
\vskip6pt
4. Check, if there exist $w_1,...,w_k\in W$ such that $w_1A_1,...,w_kA_k$ lie in a 2-plane.
\vskip6pt
5. If $w_1,...,w_k\in W$ do not exist, then $G/H$ cannot be "not C3".
\vskip6pt

 The above procedure does not require the knowledge of the embedding $\mathfrak{a}_{\h}\subset\mathfrak{a}$.
By Theorem \ref{the5} $G/H$ is of class C3, if and only if at least one $W$-orbit of the element $h_{\phi}$ does not meet $\mathfrak{a}_{\h}$. In our case $\dim\mathfrak{a}_{\h}=2$,
therefore, Step 4 of the algorithm is a necessary condition to get $G/H$ not belonging to C3. This procedure cannot be implemented directly, because the number of operations is too big (it exceeds $10^{60}$). However, we find another procedure  to settle the cases of our 4 spaces. 
\vskip6pt
\vskip6pt
Now we will describe a procedure which verifies $C3$ for homogeneous spaces in Proposition \ref{prop:4-c3}. By some abuse of terminology we will say that a subspace $P\in\mathcal{P}$ is C3, if for some $h\in\mathcal{H}$, $W\cdot h\cap P=\emptyset$.
\vskip6pt
\centerline{\bf  Procedure for verifying C3 in Proposition \ref{prop:4-c3}}
\vskip6pt
1. Compute the following sets of vectors in $\mathfrak{a}$:
$$WA_j=\{w_sA_j\,|\,w_s\in W\},\,X=\cup_{j=1}^kWA_j$$
\vskip6pt
2. Create the set $\mathcal{P}'$ of 2-planes in $\mathfrak{a}$ spanned by $A_1$ and a  vector from $X$, so
$$P\in\mathcal{P}'\Leftrightarrow P=\langle A_1,x\rangle, \comm{x\in X\setminus\{\pm A_1\}}.$$
\vskip6pt
3. For each $P\in\mathcal{P}'$ verify C3 for $P$ .
\vskip6pt
4. If a given $P\in\mathcal{P}'$ is not C3, then stop. In such case there exists $\mathfrak{h}$ reductive in $\mathfrak{g}$ such that $\operatorname{rank}_{\mathbb{R}}\mathfrak{h}=2$ and $G/H$ is not $C3$ (one can take $\mathfrak{h}:=P$).
\vskip6pt
5. If every $P\in\mathcal{P}'$ is C3, then $G/H$ is C3,  for any reductive $H$ of real rank 2.
\vskip6pt
\centerline{\bf Justification of the  procedure}
\vskip6pt
The justification is given by the following.
\begin{proposition}\label{prop:just-proc} The procedure described above  verifies C3 for any pair $(\mathfrak{g},\mathfrak{h})$ as in Proposition \ref{prop:4-c3}.
\end{proposition} 
\begin{proof}
By Proposition \ref{prop:justif} we need to verify C3 only for $\mathfrak{a}_{\h}\in\mathcal{P}$. Thus, we need to show that we can restrict ourselves to checking C3 for any $P\in\mathcal{P}'$, that is, to fix $A_1 .$ If there exists a 2-plane $V\subset \mathfrak{a}$ which meets every orbit $WA_j$, then $V$ is spanned by some $w_iA_1$ and some other $x\in X$. This is because $\operatorname{rank}_{\textrm{a-hyp}}\mathfrak{g}>1$ and because of \cite[Theorem 1.1]{ok2}. After conjugating $V$ by an element of the Weyl group $W$ we may assume that $V$ is spanned by $A_1$ and some $x'\in X$. Since our criteria of properness from Subsection \ref{subsec:kob} are invariant with respect to conjugation in $G$, the result follows.
\end{proof}

One can roughly estimate the number of operations considering the case of the biggest little Weyl group of $E_{6(2)}$. The number of $A_j$ is $14$ (see \cite[Section 8.4]{cmg}), the number of elements in the little Weyl group is $1152$. Take into account, that $A_1$ is fixed, so $|\mathcal{P}|\leq |X|-1$ (in our case, $\leq 15\cdot 1151$).

\item Let $\g$ be split of type $E_7$ or $E_8$. We also checked the maximal
  semisimple subalgebras of these Lie algebras. Again a few subalgebras could
  not be dealt with by our programs. However, we did find two examples of
  homogeneous spaces that are C2 but not C3:
$$(\mathfrak{e}_{7(7)},\mathfrak{sl}(2,\mathbb{R})\oplus\mathfrak{f}_{4(4)}), (\mathfrak{e}_{8(8)},\mathfrak{f}_{4(4)}\oplus\mathfrak{g}_{2(2)}).$$  
\end{itemize}  

\begin{center}
 \begin{table}[h]
 \centering
 {\footnotesize
 \begin{tabular}{| c | c |}
   \hline

   \hline                        
   $(\mathfrak{g},\mathfrak{h})$ & $(\mathfrak{g},\mathfrak{h})$ \\
   \hline
   $(\mathfrak{su}(3,3),\mathfrak{sl}(2,\mathbb{R})\oplus \mathfrak{su}(3))$  & $(\mathfrak{su}(2,5),\mathfrak{so}(2,5))$  \\
   \hline
   $(\mathfrak{so}(6,5),\mathfrak{sl}(2,\mathbb{R}))$  & $(\mathfrak{so}(6,7),\mathfrak{sl}(2,\mathbb{R}))$  \\
   \hline
   $(\mathfrak{sp}(4,\mathbb{R}),\mathfrak{sl}(2,\mathbb{R}))$  & $(\mathfrak{sp}(5,\mathbb{R}),\mathfrak{sl}(2,\mathbb{R}))$  \\
   \hline
   $(\mathfrak{sp}(3,3),\mathfrak{sl}(2,\mathbb{R})\oplus\mathfrak{su}(1,3))$  & $(\mathfrak{sp}(6,\mathbb{R}),\mathfrak{sl}(2,\mathbb{R}))$  \\
   \hline
   $(\mathfrak{so}(4,6),\mathfrak{so}(2,3))$  & $(\mathfrak{e}_{6(2)},\mathfrak{su}(1,2))$  \\
   \hline
	 $(\mathfrak{e}_{6(2)},\mathfrak{sl}(3,\mathbb{R}))$ & $(\mathfrak{e}_{6(2)},\mathfrak{g}_{2(2)})$  \\
	 \hline  
	\multicolumn{2}{|c|}{$(\mathfrak{f}_{4(4)},\mathfrak{sl}(2,\mathbb{R}))$} \\
   \hline 
 \end{tabular}
 }
\captionsetup{justification=centering}
\caption{Spaces with $\operatorname{rank}\mathfrak{g}\leq 6$ described in Section \ref{sec:imp}.}
\label{tab}
 \end{table}
\end{center}

\section{Proof of Theorem \ref{thm:main}}
The algorithms checking C2 and C3 for the spaces in  the database are applied after filtering it with respect to certain conditions (see Subsection \ref{subsec:filter}). Also, the database  \cite{dg} which we use yields $\mathfrak{h}\subset\mathfrak{g}$ up to an abstract isomorphism. To prove Theorem \ref{thm:main} it is sufficient to show that for every pair $(\mathfrak{g}, \mathfrak{h})$ in Table \ref{tab} the condition C2 and C3 are equivalent for any embedding $\mathfrak{h}$ into $\mathfrak{g}.$ But this can be done using Corollary \ref{cor:rank1-2}, Proposition \ref{prop:4-c3} and Theorem \ref{thm:imrn}.

\subsection{Filtering of the database}\label{subsec:filter}
We exclude some cases when considering the homogeneous spaces from the database $\mathcal{GM}$. Our filtering is based on the following theoretical results.
\begin{theorem}[\cite{bt}, Theorem 8]\label{thm:imrn} Let $G$ be a connected semisimple linear Lie group and $H$ a reductive subgroup with a finite number of components.   The following holds:
\begin{enumerate}
\item if $\operatorname{rank}_{\operatorname{a-hyp}}\mathfrak{h}=\operatorname{rank}_{\operatorname{a-hyp}}\mathfrak{g}$, then $G/H$ does not admit discontinuous actions of non-virtually abelian discrete subgroups,
\item if $\operatorname{rank}_{\operatorname{a-hyp}}\mathfrak{g}>\operatorname{rank}_{\mathbb{R}}\mathfrak{h}$ then $G/H$ admits a discontinuous action of non-virtually abelian discrete subgroup.
\end{enumerate}
\end{theorem}
\begin{proposition}\label{prop:rank1} Assume that $\operatorname{rank}_{\operatorname{a-hyp}}\mathfrak{g}>1$. Then there exist two copies of $\mathfrak{sl}(2,\mathbb{R})$ in $\mathfrak{g}$, with neutral elements $h_1,h_2\in\mathfrak{a}$ such that $\mathbb{R} h_1$ does not meet  the orbit $W(h_2)$.
\end{proposition} 
\begin{proof} It is clear that the Weyl group cannot send a half-line lying in the positive Weyl chamber to other half-line in the positive Weyl chamber. By \cite[Theorem 1.1]{ok2}, $\mathfrak{b}$ is spanned by neutral elements (contained in $\mathfrak{b}^{+}$) of some copies of $\mathfrak{sl}(2,\mathbb{R})$ in $\mathfrak{g}.$

\end{proof}
\begin{corollary}\label{cor:rank1-2}
If $\operatorname{rank}_{\mathbb{R}}\mathfrak{h}=1<\operatorname{rank}_{\operatorname{a-hyp}}\mathfrak{g}$ then $(\mathfrak{g},\mathfrak{h})$ is C3.
\end{corollary}
\begin{proof}
Assume that $(\mathfrak{g},\mathfrak{h})$ is not C3. Take $h_{1}, h_{2}\in\mathfrak{a}$ given in Proposition \ref{prop:rank1}. By Theorem \ref{tthh} there exists $w_{1}\in W$ such that $\mathbb{R}w_{1}h_{1}=\mathfrak{a}_{\h}$ (as $\mathfrak{a}_{\h}$ is 1-dimensional) and $w_{2}\in W$ such that $w_{2}h_{2}\in \mathfrak{a}_{\h}$. But this contradicts Proposition \ref{prop:rank1}.
\end{proof}

\noindent The Calabi-Markus phenomenon and Theorem \ref{thm:imrn}  eliminate the necessity to consider the cases
\begin{enumerate}
\item $\operatorname{rank}_{\mathbb{R}}\mathfrak{g}=\operatorname{rank}_{\mathbb{R}}\mathfrak{h}$,
\item $\operatorname{rank}_{\textrm{a-hyp}}\mathfrak{g}=\operatorname{rank}_{\textrm{a-hyp}}\mathfrak{h}$
\end{enumerate}
since they are not C2. 
Finally, we eliminate the case $\operatorname{rank}_{\mathbb{R}}\mathfrak{h}=1<\operatorname{rank}_{\textrm{a-hyp}}\mathfrak{g}$. This follows from Corollary \ref{cor:rank1-2}.
 \begin{remark}{\rm Looking for $G/H$ which are C2 but not C3, when $\operatorname{rank}\mathfrak{g}>6$ we consider the cases
$$\operatorname{rank}_{\operatorname{a-hyp}}\mathfrak{g}>\operatorname{rank}_{\mathbb{R}}\mathfrak{h}$$
 since  C2 is satisfied by Theorem \ref{thm:imrn}, and  C3 is checked by the general procedure.}
\end{remark}

\subsection{Completion of the proof of Theorem \ref{thm:main}}
We complete the proof of Theorem \ref{thm:main} running the computer program which checks the remaining cases in the database. 
\vskip10pt
\noindent{\bf Acknowledgment}. The first named author was supported by the National Science Center, Poland, grant NCN, no.   2018/31/D/ST1/00083. The fourth named author was supported by the National Science Center, Poland, grant no.  2018/31/B/ST1/00053. 

We cordially thank the referees for valuable advice.

\vskip20pt
Department of Mathematics,

\noindent University of Trento,

\noindent Via Sommarive 14,

\noindent I-38123 Povo (Trento), Italy

\noindent e-mail address: degraaf@science.unitn.it
\vskip20pt
\noindent Faculty of Mathematics and Computer Science

\noindent University of Warmia and Mazury

\noindent S\l\/oneczna 54, 10-710 Olsztyn, Poland

\noindent e-mail adresses:

\noindent mabo@matman.uwm.edu.pl (MB)

\noindent piojas@matman.uwm.edu.pl (PJ)

\noindent tralle@matman.uwm.edu.pl (AT)

\end{document}